\documentclass[reqno, 12pt, reqno]{amsart}

\usepackage{amsfonts, amsthm, amsmath, amssymb}

\usepackage{hyperref}

\usepackage{graphicx}

\usepackage[margin=1.5in]{geometry}

\RequirePackage{mathrsfs} \let\mathcal\mathscr

\numberwithin{equation}{section}

\newtheorem{theorem}{Theorem}[section]
\newtheorem{lemma}[theorem]{Lemma}

\theoremstyle{definition}
\newtheorem*{ack}{Acknowledgements}

\newtheorem{rem}[theorem]{Remark}
\newtheorem{definition}[theorem]{Definition}

\renewcommand{\d}{\mathrm{d}}
\renewcommand{\phi}{\varphi}

\newcommand{\0}{\mathbf{0}}

\newcommand{\PP}{\mathbb{P}}
\renewcommand{\AA}{\mathbb{A}}
\newcommand{\A}{\mathbf{A}}
\newcommand{\FF}{\mathbb{F}}
\newcommand{\ZZ}{\mathbb{Z}}

\newcommand{\NN}{\mathbb{N}}
\newcommand{\QQ}{\mathbb{Q}}
\newcommand{\RR}{\mathbb{R}}
\newcommand{\CC}{\mathbb{C}}

\newcommand{\cO}{\mathcal{O}}
\newcommand{\cL}{\mathcal{L}}

\newcommand{\Gal}{{\rm Gal}}

\renewcommand{\leq}{\leqslant}

\renewcommand{\geq}{\geqslant}

\renewcommand{\bar}{\overline}

\renewcommand{\a}{\mathbf{a}}

\newcommand{\ve}{\varepsilon}

\DeclareMathOperator{\rank}{rank}
\DeclareMathOperator{\Proj}{Proj}
\DeclareMathOperator{\disc}{disc}

\DeclareMathOperator{\Pic}{Pic}

\DeclareMathOperator{\NS}{NS}
\DeclareMathOperator{\Image}{Im}

\renewcommand{\t}{\mathbf{t}}

\newcommand{\eeq}{\end{equation}}
\newcommand{\beql}[1]{\begin{equation}\label{#1}}

\begin{document}

\title[The Cayley ruled cubic]{Counting rational points on the\\ Cayley ruled cubic}
\author{R. de la Bret\`eche}
\author{T.D.\ Browning}
\author{P. Salberger}

\address{
Institut de Math\'ematiques de Jussieu 
--- Paris Rive Gauche\\
 UMR 7586\\
Universit\'e Paris Diderot\\ 
B\^atiment Sophie Germain\\ 
75205 Paris cedex 13\\ France}
\email{regis.de-la-breteche@imj-prg.fr}

\address{School of Mathematics\\
University of Bristol\\ Bristol\\ BS8 1TW\\ UK}
\email{t.d.browning@bristol.ac.uk}

\address{Chalmers University of
    Technology\\
G\"oteborg SE-412 96\\ Sweden}
\email{salberg@chalmers.se}

\date{\today}


\begin{abstract}
We count rational points of bounded height on the Cayley ruled cubic surface and interpret the result in the context of general conjectures due to  Batyrev and Tschinkel.
\end{abstract}

\maketitle
\setcounter{tocdepth}{1}
\tableofcontents

\thispagestyle{empty}

\section{Introduction}

\maketitle

The arithmetic of singular cubic surfaces $S\subset \PP^3$ 
has long  been the subject of intensive study. When $S$ is defined over $\QQ$ and has  
isolated ordinary singularities then 
the set $S(\QQ)$ of
rational points on $S$ is Zariski dense in $S$ as 
soon as it is non-empty. 
Under this hypothesis, a finer measure of density is achieved by studying the counting function
$$
N(U;B)=\#\{t \in U(\QQ): H(t) \leq B\},
$$
where $H:S(\QQ)\to \RR_{>0}$ is an anticanonical height function and $U\subset S$ is obtained by deleting the lines from $S$.   

The conjectures of Manin \cite{f-m-t} and Peyre \cite{Pey03} give a precise
prediction for the asymptotic behaviour of $N(U;B)$, as  $B \rightarrow \infty$, 
for normal del Pezzo surfaces in terms of certain invariants associated
to a minimal resolution. The conjecture has now been resolved for several 
singular cubic surfaces over $\QQ$. Most recently, for example,  Le Boudec \cite{boudec}
has handled a 
cubic surface with singularity type 
 $\mathbf{D}_4$ (see the references therein for earlier work on this topic). 
However, the conjectures of Manin and Peyre offer no prediction
for cubic surfaces with non-isolated singularities. Indeed, the asymptotics
for such surfaces are different as they contain infinitely many lines.

The primary goal of this paper is to study the counting function for a particular non-normal cubic surface and to show that the resulting asymptotic formula can still be interpreted in the context of a much more general suite of conjectures due to Batyrev and Tschinkel \cite{BT}.  
According to 
Dolgachev \cite[Thm~9.2.1]{dolt}, any  irreducible non-normal cubic surface over $\QQ$ is either 
a cone over an irreducible singular plane cubic, or it is projectively equivalent to one of the (non-isomorphic) 
surfaces
\begin{equation}\label{eq:other-one}
t_0^2t_2-t_1^2t_3=0
\end{equation}
or 
\begin{equation}\label{eq:cayley}
t_0t_1t_2-t_0^2t_3-t_1^3=0,
\end{equation}
both of which are singular along the line $t_0=t_1=0$.
 These surfaces arise as different projections of the cubic scroll in $\PP^4$,  which is isomorphic to  the (ruled) {\em Hirzebruch surface} $\FF_1$ (i.e. a del Pezzo surface of degree $8$).

\begin{figure}
\begin{center}
\includegraphics[scale=0.15]{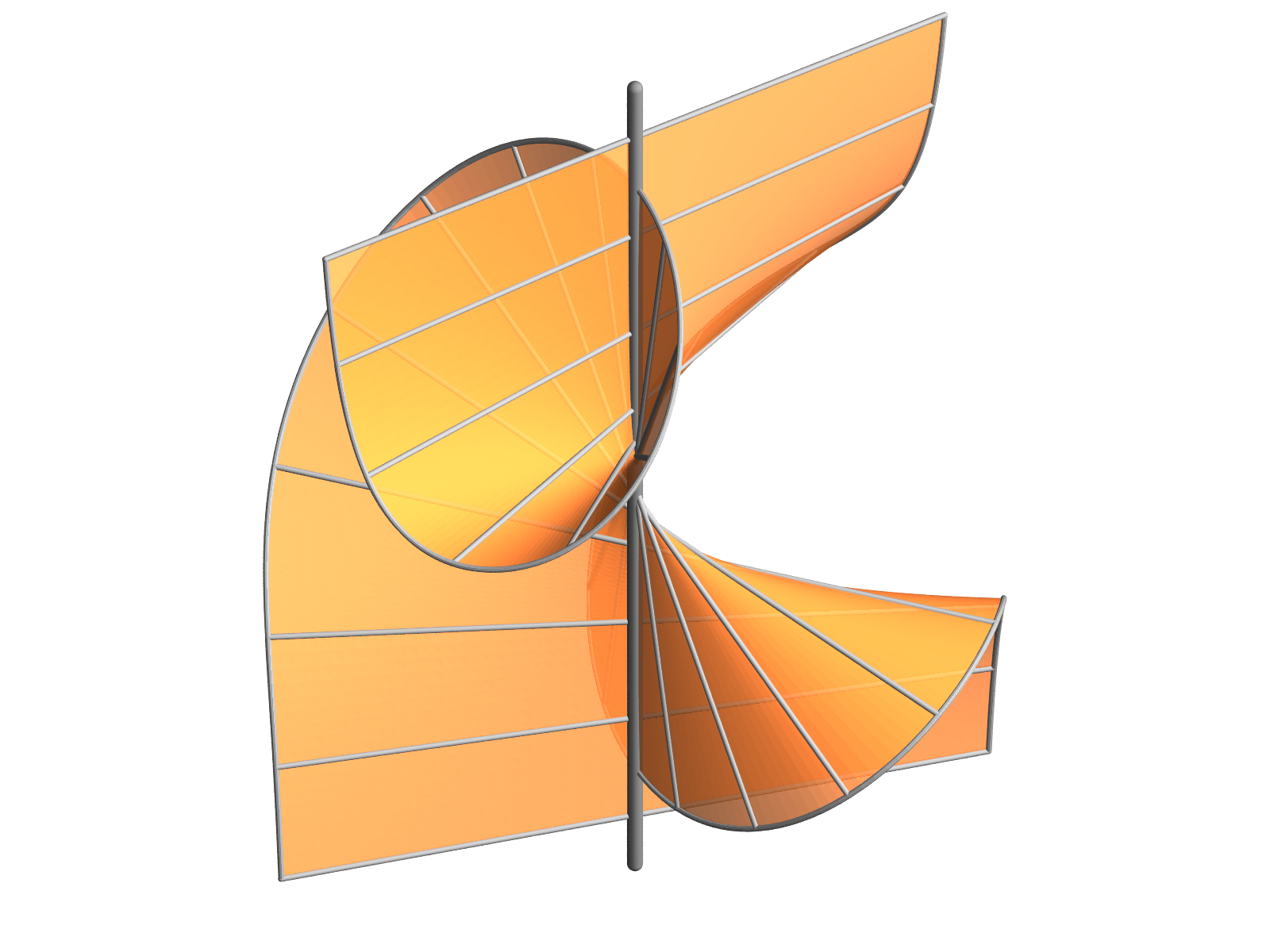}
\end{center}
\caption{The Cayley ruled cubic surface}\label{fig}
\end{figure}

For the remainder of this paper we will focus exclusively on the cubic surface \eqref{eq:cayley}, illustrated in Figure \ref{fig}.
This is called the 
{\em Cayley ruled surface} and we will denote it by 
 $W\subset \PP^3$. 
While \eqref{eq:other-one} is plainly toric the Cayley surface is not toric. Indeed, 
according to Gmeiner and Havlicek \cite[Lemma 3.1]{cayley}, the  automorphism group of $W$ is a 3-dimensional algebraic group, which contains a  
2-dimensional unipotent subgroup. Thus there is no $2$-dimensional torus acting faithfully on $W$.

Let $V=W\setminus\{t_0=t_1=0\}$  be the complement of the double line in $W$. 
Clearly $V\cong \AA^2$.
 Finally, we take our height function $H:V(\QQ)\to \RR_{>0}$ to be 
 metrized by the Euclidean norm. (i.e. $H(t)=\|\t\|:=\sqrt{t_0^2+\dots+t_3^2}$ if 
$t$ is represented by a primitive vector $\t\in \ZZ_{\mathrm{prim}}^4$.)
It then follows from a computation of Serre 
\cite[\S 2.12]{serre} that $N(V;B)=O_{V} (B^{2})$.
We are able to establish a precise  asymptotic formula, as follows.

\begin{theorem}\label{t:main}
We have 
$$
N(V;B)=
\frac{\pi B^2}{2\zeta(2)} \sum_{\substack{(\lambda,\mu)\in \ZZ_{\mathrm{prim}}^2\\
\mu\neq 0}} \frac{1}{\sqrt{f(\lambda,\mu)}}
+O(B^{3/2}\log B),
$$
where $f(\lambda,\mu)=\lambda^6+2\lambda^4\mu^2+\lambda^2\mu^4+\mu^6$. 
\end{theorem}

Since $W$ is not toric this result is not implied by work of Batyrev and Tschinkel \cite{BTa}.
In \S \ref{s:BT}  we will prove  that this result is compatible with some very general conjectures of Batyrev and Tschinkel 
\cite{BT}
about ``weakly $\mathcal{L}$-saturated'' smooth quasi-projective varieties.
The first step involves constructing an explicit desingularisation of $W$, which we record here for the sake of convenience. 

\begin{theorem}\label{t:resolve}
Let $X\subset \PP^2\times \PP^1$ be the biprojective surface with coordinates $(x_0, x_1, x_2; y_1, y_2)$ defined by $x_1y_2=x_2y_1$.  Then the morphism $\phi: X \to W$ defined by
$$
\phi(x_0, x_1, x_2; y_1, y_2)= (x_1y_1 , x_1y_2 , x_0y_1+ x_2y_2, x_0y_2)
$$
is a desingularisation of $W$ such that the open subvariety $U$ of $X$ where $x_1\neq 0$ is sent 
isomorphically onto the subset $V$ of W where $t_0\neq 0$.
\end{theorem}

The surface $X$ is isomorphic to $\FF_1$ and it is also the  normalisation of $W$ (see Remark \ref{rem:n}). 
Despite starting with an anticanonical counting problem for the singular cubic surface $W$, Theorem \ref{t:resolve} leads 
to a counting problem for the non-singular surface $X$, 
endowed with an ample but  {\em non}-anticanonical linear system. 
For $m>1$, Billard \cite{billard} has provided precise asymptotics for counting functions associated to the Hirzebruch surface $\FF_m$ 
endowed with a {\em general} complete linear system.  For $m=1$, the case of primary interest to us,
work of Chambert-Loir and Tschinkel \cite[Thm.~4.16]{CL-T} handles the corresponding counting problem  
associated to a particular choice of metric.

\medskip

We will offer two very different  proofs of Theorem \ref{t:main}.
It should be emphasised that both methods are capable of producing asymptotic formulae for counting functions associated to other non-normal surfaces. Handling the
cubic surface 
\eqref{eq:other-one}, for example, is  easier than  $W$ and 
leads to similar asymptotic behaviour.

 The simplest proof of Theorem \ref{t:main} is found in \S \ref{s:first}. It relies on an explicit realisation of the Fano variety
 $F_1(W)\subset \mathbb{G}(1,3)\subset \PP^5$,  parametrising lines on  $W$, as  the union of an isolated point and a twisted cubic.  A standard result from the geometry of numbers is then invoked to handle the contribution from the rational points on the lines. 
 
The second approach is  found in \S \ref{s:second}.
It uses the fact that $W$ is an equivariant compactification of the additive algebraic group $\mathbb{G}_a^2$, in order to
study 
the analyticity  of the associated  height zeta function using adelic Poisson summation. 
This argument  is modelled on the methods of Chambert-Loir and Tschinkel \cite[\S 3]{CL-T},
which were developed to study equivariant compactifications of vector groups. A noteworthy
 feature of the proof    is that we get contributions to the main term from some of the non-trivial characters.
The counting 
 function $N(V;B)$ 
can be  interpreted as a counting function on $X$ 
endowed with an ample line bundle of bidegree $(1,1)$ and a certain metric  which is inherited from the singular model $W$ (see \S \ref{s:BT}).  This counting function is related to the counting function on $X$ considered 
in \cite[Thm.~4.16]{CL-T}, but the latter does not imply Theorem \ref{t:main} since it involves a different metric.

\begin{rem}
Although we are concerned here with rational points on $W$,  the problem of counting integer points on any affine model  is also of interest.  For  either of the affine   surfaces 
$xyz=x^2+y^3$ or $xy=x^2z+y^3$  it is possible to  show that the number of integers $(x,y,z)\in (\ZZ\cap [-B,B])^3$ has order of magnitude $B$.  This is in agreement with the {\em affine surface hypothesis} proposed in \cite{pila}.
\end{rem}

\begin{ack} 
While working on this paper the 
 first author was supported by an {\em IUF Junior} and the
second author
was  supported by {\em ERC grant} \texttt{306457}. 
The authors are very grateful to Professor Hans Havlicek for allowing us to include  Figure \ref{fig} which was created by him, and to the anonymous referee for some useful comments.
\end{ack}

\section{The Batyrev--Tschinkel conjecture}\label{s:BT}

Let us begin by establishing Theorem \ref{t:resolve}.
Let $\pi: X\to  \PP^2$ be 
the projection from $(x_0, x_1, x_2; y_1, y_2)$ to $(x_0, x_1, x_2)$ and let $O_1\subset  \PP^2$
be the
open subset where $x_1\neq 0$. Then $\pi$ 
restricts to an isomorphism $\pi_1:U\to O_1$.
Next, let $O\subset  \PP^2$ be the open subset where $(x_1, x_2)\neq (0, 0)$. 
There is then a morphism $f: O\to V$ defined by $t_i = Q_i(x_0, x_1, x_2)$,  for $0\leq i\leq 3$, where
\begin{align*}
Q_0(x_0, x_1, x_2) &= x_1^2,  \\             
Q_1(x_0, x_1, x_2) &= x_1x_2,  \\                  
Q_2(x_0, x_1, x_2) &= x_0x_1+x_2^2,  \\     
Q_3(x_0, x_1, x_2) &= x_0x_2.
\end{align*} 
This morphism restricts to an isomorphism $f_1: O_1\to V$,  with corresponding inverse $V\to O_1$  such that 
$(1, t_1/t_0 , t_2/t_0 , t_3/t_0)$  is sent to  
$$
(x_0/x_1 ,1, x_2/x_1)= (-(t_1/t_0)^2+t_2/t_0,1, t_1/t_0).
$$
Since $\phi=f_1\circ \pi_1 $ on $U$, it follows that $\phi$ restricts to an isomorphism $\phi: U\to V$, as desired.

\begin{rem}\label{rem:n}
The morphism $\phi: X\to W$ is finite since it is projective and quasi-finite (see \cite[Ex.~III.11.2]{hart}). 
Since $\phi$ is birational, furthermore,  it is therefore the normalisation of $W$ (see \cite[Ex.~12.20]{GW}).
\end{rem}

We now proceed to recast the counting function $N(V;B)$ in the language of adelic metrics.
Let  $|\cdot|_p$ be the usual absolute value on  $\QQ_p$ defined by 
$|p^\nu x|_p=p^{-\nu}$ if  $\nu\in \ZZ$ and  $x\in U_p=\ZZ_p^*$. 
Let $M=\cO_W(1)$ and let $s_0,\dots,s_3$ be the global sections of $M $ given by the coordinates $t_0,\dots,t_3$
of $\PP^3$. We may then define a $p$-adic norm 
$\|\cdot\|_p$ on $M$ by 
$$
  \|s(w_p)\|_p=\min_i |(s/s_i)(w_p)|_p,
$$
for a local section  $s$ of $M$ at a point $w_p\in W(\QQ_p)$ and where $i\in \{0,1,2,3\}$  
runs over the global sections $s_i$ such that $s_i(w_p)\neq 0$.  
At the archimedean place we define a real norm $\|\cdot \|_\infty$ on $M$ by
\begin{equation}\label{eq:arch}
\|s(w_\infty)\|_\infty=\left(\sum_i |s_i/s(w_\infty)|^2\right)^{-1/2}
\end{equation}
for a local section  $s\neq 0$ of $M$ at a point $w_\infty\in W(\RR)$.

Now let $\|\cdot\|_v$ denote 
$\|\cdot\|_\infty$ or $\|\cdot\|_p$ for a prime $p$.
Then we get an adelic metric $(\|\cdot\|_v)$  on $M$ as in Peyre \cite{Pey95} and a height on $W(\QQ)$ defined by
$$
H(w)=\prod_v \|s(w)\|_v^{-1},
$$
for a rational point $w$ on $W$ and a local section $s$ of $M$ with $s(w)\neq 0$. This height does not depend on the choice of $s$. For a rational point $P$ on $V$  represented by $(1, t_1, t_2, t_3)$, we may (for example) choose $s$ to be $s_0$, which gives
$$
 H(P)=\sqrt{1+ t_1^2+ t_2^2+ t_3^2}\prod_p \max\{1, |t_1|_p ,
 |t_2|_p, |t_3|_p\}.
 $$
We are then interested in the counting function
$$N(V;B)= \#\{P\in V(\QQ) : H(P)\leq B\}.$$
The  main goal of this section is to give an  explicit description of what the conjectures of Batyrev and Tschinkel
\cite{BT}
 predict for the asymptotic behaviour  of $N(V; B)$, as $B\rightarrow \infty$.

\medskip

For $k>0$ and a place $v$ of $\QQ$, there exists a $v$-adic norm 
$\|\cdot \|_{k,v}$ on $M^{\otimes k}$
such that 
$$\| s^k(w_v) \|_{k,v}=
\| s(w_v) \|_{v}^k$$ for any 
local section $s$ of $M$ at a point $w_v\in W(\QQ_v)$. 
For $k=0$, let $M^{\otimes k} = \cO_W$ and denote by
$\| \cdot \|_{0,v}$ the trivial metric given by 
$\| g(w_v) \|_{0,v}=|g(w_v)|_v$ for a local continuous function $g : N_v\to  \QQ_v$ defined on an open $v$-adic analytic neighbourhood of $N_v\subset  W(\QQ_v)$.  For $k \not\in \{1,2\}$ we shall  only consider the $v$-adic norm
$\| \cdot \|_{k,v}$ at the archimedean place $v=\infty$, where $\QQ_v=\RR$. In this setting we will  use the formula 
\eqref{eq:arch} to define a norm on $M$  for complex points $w_\infty\in W$ and then extend the above definition of power norms 
$\| \cdot \|_{k,\infty}$ to complex points on $W$. 

Now let
$L=\cO_V(1)$ be
the restriction of $M=\cO_W(1)$ to the open subset $V\subset W$ and let 
$L^{\otimes k}=\cO_V(k)$ for $k\geq 0$.
Then, for $k\geq 0$, 
 $\cL=(L, \|\cdot\|_\infty)$  and 
$\cL^{\otimes k}=(L^{\otimes k}, \|\cdot\|_{k,\infty})$
are {\em metrized invertible sheaves} in the notation of \cite[Def.~2.1.1]{BT}.

\begin{definition}
Let $H^0_{\mathrm{bd}}(V, \cL^{\otimes k})$ be the set of $s\in H^0(V,M^{\otimes k})$ for which $\| s \|_{k,\infty}$ is bounded on $V(\CC)$. Let $A(V, \cL)= \bigoplus_{k\geq 0}  H^0_{\mathrm{bd}}(V, \cL^{\otimes k})$.
\end{definition}

Next we recall that $\phi:X\to W$ restricts to an isomorphism  $U\to V$.  Thus  there is  a natural restriction map from $H^0(X, (\phi^* M)^{\otimes k})$ to 
$$
H^0(U, (\phi^* M)^{\otimes k})=H^0(V, M^{\otimes k}),
$$ 
for each $k\geq 0$.
The following result (and its proof) is essentially a specialisation of  \cite[Prop.~2.1.3]{BT}
to the Cayley cubic. 

\begin{lemma}
The image of the restriction map from 
$H^0(X, (\phi^* M)^{\otimes k})$
to 
$H^0(V, M^{\otimes k})$
is equal
to 
$H^0_{\mathrm{bd}}(V, \cL^{\otimes k})$.
\end{lemma}

\begin{proof}
 The inclusion $\Image
 H^0(X, (\phi^* M)^{\otimes k}) \subset H_{\mathrm{bd}}^0(V, \cL^{\otimes k})$ 
 follows from the compactness of $X(\CC)$ as in 
 \cite[Prop.~2.1.3]{BT}.
Conversely, if we regard $s\in H_{\mathrm{bd}}^0(V, \cL^{\otimes k})$ as an element of 
$H^0(U, (\phi^* M)^{\otimes k})$
and let $s_i \in H_{\mathrm{bd}}^0(V, \cL)$ 
correspond to $t_i$, then there exists $K >0$ such that $\min_i |s/s_i^k|< K$ on $X(\CC)= \CC^2$. 
For 
 $0\leq i\leq 3$,  let $X_i$ be the open subset of $X$ where $\phi^{-1}(t_i)\neq 0$ 
and let $U_i$ be the open subset of $U\cap X_i$  where $|s/s_i^k|< K$. Then the bounded holomorphic function $s/s_i^k$ on $U_i$ extends uniquely to a bounded holomorphic function $h_i$ on $X_i$ by the first extension theorem of Riemann (see \cite[p.~38]{FG}). The local analytic sections $h_is_i^k$ on $X_i$  will glue to a global analytic section $\tilde s$ of $(\phi^*M)^{\otimes k}$ on $X$, which is algebraic by \cite[Appendix B.4]{hart}. Since $\tilde  s$  restricts to $s$ on $V$, we get that  $s\in  
\mathrm{Im}~
 H^0(X, (\phi^* M)^{\otimes k})$ and we are done.
\end{proof}

From this result  we immediately obtain the following result.

\begin{lemma}\label{lem:S2}
There is a natural isomorphism of graded rings between $A(V, \cL)$ and
$\bigoplus_{k\geq 0}  
H^0(X, (\phi^* M)^{\otimes k})$.
In particular, $A(V, \cL)$ is finitely generated.
\end{lemma}

Batyrev and Tschinkel call $\Proj A(V, \cL)$ the {\em $\cL$-primitive closure of $V$} (see \cite[Def.~2.1.6]{BT}). 
Apart from depending on $V$ and $L=\cO_V(1)$, it also   depends on the restriction of the complex norm 
$\|\cdot \|_\infty$ on $M$  to $L$. 
The line bundle $\phi^*M$ is very ample of bidegree $(1,1)$ on $X\subset \PP^2\times \PP^1$ and it embeds $X$ into $\PP^4$ as a cubic scroll. 
But it is well-known that a cubic scroll in $\PP^4$
is projectively normal 
(cf. \cite{O} and \cite[Ex.~II.5.14]{hart}), whence 
Lemma \ref{lem:S2} allows us to identify $\Proj A(V, \cL)$ with $X$. This  is  important for us, since  the conjectures about  $N(V; B)$ in \cite{BT} are formulated in terms of the geometry of $\Proj A(V, \cL)$.

It follows from the proof of Theorem \ref{t:main} in \S \ref{s:first} that the main term receives contributions
from infinitely many lines. Thus, for any Zariski locally closed subset $Z\subset V$ with $\dim Z< \dim V=2$ we have  
$$
\lim_{B\rightarrow \infty} \frac{N(Z;B)}{ N(V; B)} <1.
$$ 
This means that  $V$ is {\em weakly $\cL$-saturated}  (see \cite[Def.~3.2.2]{BT}).
Similar reasoning shows that $V$ contains no {\em strongly $\cL$-saturated} Zariski dense open subset (see 
\cite[Def.~3.2.3]{BT}).  

We
 recall the definition of the invariant 
$a_\cL(V)$
from \cite[Def.~2.2.4]{BT}. It is the infimum of all $t\in \QQ$ such that the class of $t[\phi^*L]+ [K_X]$ is in the effective cone of the N\'eron--Severi space $\NS(X)_\RR$. But $X$ is the blow-up of $\PP^2$ in a point and it is well known that $\NS(X)=\Pic (X)= \ZZ^2$ and that the restriction from $\Pic(\PP^2\times \PP^1)$ to 
$\Pic( X)$ is an isomorphism. Since the anticanonical sheaf of $\PP^2\times \PP^1$ is of bidegree $(3, 2)$ and $X\subset \PP^2\times\PP^1$ is given by a bilinear equation, the anticanonical sheaf on $X$ must have  bidegree $(2,1)$. Hence
$a_\cL(V)=2,$
since $[\phi^*L]$ has bidegree $(1,1)$.

We may now refer to \cite[\S 3.5]{BT} to obtain a conjecture for the asymptotic growth of
$N(V; B)$. Since $a_\cL(V)[\phi^*L]+ [K_X]$ has bidegree $(0,1)$ in $\Pic( X)$, it is represented by the class  $[D]$ of a fibre $D$ of the projection $f$ from $X\subset \PP^2\times \PP^1$ to $Y=\PP^1$. This means that $D$ is not {\em rigid}
(see \cite[Def.~2.3.1]{BT}) and so  $V$ is not {\em $\cL$-primitive} in the sense of \cite[Def.~2.3.4]{BT}. 
We therefore find ourselves in Case  1 of \cite[\S 3.5]{BT} and,  as expected, there is an $\cL$-primitive fibration given by the projection $f: X\to Y$. The fibres $X_y=f^{-1}(y)$ of $f$ are lines on $\PP^2$ and 
give the lines 
$V_y=\phi(U\cap X_y)$ on $W$, with defining equations
\begin{equation}\label{eq:defining}
 \lambda t_0-\mu t_1=\lambda\mu t_2-\lambda^2t_1-\mu^2 t_3=0,
\end{equation}
where $(y_1,y_2)=(\lambda,\mu)$ are the homogeneous coordinates representing the point $y$ on $Y=\PP^1$.
In fact the lines $V_y$ are parametrised by points $y$ on the open subset $Y_0=\AA^1$ of $Y$ where $y_2\neq 0$. 
Each rational point on $Y_0$ is represented by exactly 
two points $(\lambda,\mu)\in \ZZ_{\mathrm{prim}}^2$ with $\mu\neq 0$.

It is now easy to calculate the invariants  $a_\cL(V_y)$
and $\beta_\cL(V_y)$ for $V_y$. These are given by $a_\cL(V_y) =2= a_\cL(V)$  and $\beta_\cL(V_y)=\rank \Pic( X_y)=1$. The conjecture of Batyrev and Tschinkel therefore  predicts that 
\begin{equation}\label{eq:predict}
N(V;B)= c_\cL(V)B^2+o(B^2),
\end{equation}
as $B\rightarrow \infty$, where 
 $c_\cL(V)$ is a sum of constants $\sum_{y\in Y_0(\QQ)} c_\cL(V_y)$.
The constant  $c_\cL(V_y)$ is given by
$$
c_\cL(V_y) = \frac{\gamma_\cL(V_y)\delta_\cL(V_y)\tau_\cL(V_y)}{ a_\cL(V_y) (\beta_\cL(V_y)-1)!}=\frac{\gamma_\cL(V_y)\tau_\cL(V_y)}{2},
$$
since $\delta_\cL(V_y)= \# H^1(\Gal(\bar{\QQ}/ \QQ), \Pic(\bar{X_y}) )=1$.
The $\gamma$-invariant is the same as Peyre's $\alpha$-invariant that was introduced in \cite{Pey95}, since
 $\rank \Pic( \bar{X_y}) =1$ 
(for the comparison see \cite[p.~335]{Pey03}). According to \cite[p.~150]{Pey95}, therefore, we  have  $\gamma_\cL(V_y)= 
\alpha (X_y)=\tfrac{1}{2}$.

In order to compute $\tau_\cL(V_y)$, we 
 make use of the fact that
 $\tau_\cL(V_y)$
coincides with 
the Tamagawa constant 
 $\tau_\cL(X_y)$, 
defined by  Peyre  \cite[p.~119]{Pey95}.
To define the latter, let $\phi_y:X_y\to W$ be the restriction of $\phi:X\to W$ to $X_y$ and let 
$\|\cdot \|_{k,v}'$
be the pullback norm of $\|\cdot \|_{k,v}$ on $\phi_y^*(M^{\otimes k})$ (cf. \cite[p.~100]{salberger}).  Furthermore, 
in the light of  \eqref{eq:defining}, 
we let $\tau_0,\tau_1$ be homogeneous coordinates for $X_y=\PP^1$ such that $\phi_y(\tau_0,\tau_1)=(t_0,t_1,t_2,t_3)$, with $y=(\lambda,\mu)\in Y_0$ and 
(as in the proof of  Lemma~\ref{lem:lines})
$$
t_0=\mu^2 \tau_0,\quad 
t_1=\lambda \mu \tau_0,\quad
t_2=\lambda^2\tau_0+\mu \tau_1,\quad
t_3=\lambda \tau_1.
$$
This expresses $t_i$, for each $0\leq i\leq 3$, as a linear function $L_i(\tau_0,\tau_1)$, say.
Let  $(\sigma_0,\sigma_1)$ be the global sections of $\phi_y^*(M)$ corresponding to the homogeneous coordinates $\tau_0,\tau_1$ for $X_y$. We then have 
$$
\|\sigma(x_p)\|_{2,p}'=\min\left\{  
|(\sigma/\sigma_0^2)(x_p)|_p,  |(\sigma/\sigma_1^2)(x_p)|_p
\right\}
$$
for a local section $\sigma$ of $\phi_y^*(M^{\otimes 2})=\cO_{\PP^1}(2)$ at a point $x_p\in X_y(\QQ_p)$, while
$$
\|\sigma(x_\infty)\|_{2,\infty}'=
\|\sigma(x_\infty)\|_{1,\infty}'=
\left(
\sum_{0\leq i\leq 3} \left( L_i(\sigma_0,\sigma_1)^2/\sigma\right)(x_\infty)
\right)^{-1}
$$
for a local section $\sigma$ of $\phi_y^*(M^{\otimes 2})$ 
with $\sigma(x_\infty)\neq 0$ at 
 a point $x_\infty\in X_y(\RR)$.
Equipped  with these facts we are now ready to calculate 
the value of $\tau_\cL(V_y)$.

\begin{lemma}\label{lem:density}
For $y\in Y_0(\QQ)$ we have 
$$
\tau_\cL(V_y)
=\frac{2\pi}{\zeta(2)\sqrt{f(\lambda,\mu)}},
$$
where $f(\lambda,\mu)$ is as in the statement of Theorem \ref{t:main}.
\end{lemma}

\begin{proof}
The $v$-adic norms 
$\|\cdot\|_{2,v}'$ on the anticanonical sheaf $\cO_{\PP^1}(2)$ give rise to measures $\omega_v$ on $X_y(\QQ_v)$ 
(see  \cite[p.~112]{Pey95}) and a product measure 
 $\omega_{\A_\QQ}$
on the ad\`eles $X_y(\A_\QQ)= \prod_v X_y(\QQ_v)$.
The definition of $\omega_{\A_\QQ}$ requires the  convergence factors 
$L_p(1, \Pic (\bar{X_y}))$, which in this case are equal to $(p-1)/p$
for all $p$.  Hence 
$$
\tau_\cL(V_y)= 
 \omega_{\A_\QQ}(X_y(\AA_\QQ))
 =\omega_\infty (X_y(\RR))
\prod_p \left(1-\frac{1}{p}\right) 
\omega_p (X_y(\QQ_p)).
$$
The proof  of \cite[Lemme 2.2.1]{Pey95} shows that
$$
\omega_p(X_y(\QQ_p))=\frac{\#X_y(\FF_p)}{p}=\frac{p+1}{p}
$$ 
for all primes $p$, whence
$$
\tau_\cL(V_y)
=\frac{\omega_\infty (X_y(\RR))}{\zeta(2)}.
$$
It remains to compute the volume $\omega_\infty (X_y(\RR))$.

According to the  definition of measure $\omega_\infty$ in \cite[p.~112]{Pey95},
we need to compute the volume for the real measure on 
$X_y(\RR)=\PP^1(\RR)$ associated to the real norm $\|\cdot\|_{2,\infty}'$ on $\cO_{\PP^1}(2)$.
This measure may be viewed 
as the Riemannian density  
(see \cite[p.~136]{GHL}, for example) associated to  the  Riemannian metric on 
$X_y(\RR)$ that one obtains by pulling back the standard 
Riemannian metric on $\PP^3(\RR)= S^4/\ZZ^2$ along the embedding 
$\psi_y:X_y(\RR)\to \PP^3(\RR)$, given by 
$\phi_y$ and $W(\RR)\subset \PP^3(\RR)$.

If we let $u$ be the affine coordinate $\sigma_1/\sigma_0=\tau_1/\tau_0$ for $X_y$ and 
$$
Q(u)=\sum_{0\leq i\leq 3} L_i(1,u)^2=(\lambda^2+\mu^2)u^2+2\lambda^2\mu u+\lambda^4+\lambda^2\mu^2+\mu^4,
$$
then \cite[Eq.~(2.2.1)]{Pey95} implies that 
$\omega_\infty$ is the measure $\d u/Q(u)$ on the open subset of $X_y$ where $\tau_0\neq 0$.
It therefore follows that 
$$
\omega_\infty(X_y(\RR))=\int_{-\infty}^\infty \frac{\d u}{Q(u)}=\frac{2}{\sqrt{\disc(Q)}}\int_{-\infty}^\infty \frac{\d u}{u^2+1}=
\frac{2\pi}{\sqrt{f(\lambda,\mu)}},
$$
as required to complete the proof of the lemma. 
\end{proof}

This completes our calculation of the constant $c_{\cL}(V)$ in \eqref{eq:predict}. Combining Lemma 
\ref{lem:density} with the preceding discussion we conclude that 
$$
c_{\cL}(V)
=\frac{\pi}{4\zeta(2)} \sum_{(\lambda,\mu)\in Y_0(\QQ)} \frac{1}{\sqrt{f(\lambda,\mu)}}
=\frac{\pi}{2\zeta(2)} \sum_{\substack{(\lambda,\mu)\in \ZZ_{\mathrm{prim}}^2\\ \mu\neq 0}} \frac{1}{\sqrt{f(\lambda,\mu)}},
$$
which aligns perfectly with the statement of Theorem \ref{t:main}.

\section{First approach: using the lines}\label{s:first}

The 
Fano variety of lines
 $F_1(W)\subset \mathbb{G}(1,3)$ on $W$  is the union of an isolated point and a twisted cubic. The former component corresponds to the double line $\{t_0=t_1=0\}$ and the latter corresponds to the family of lines
$$
V_{y}=\left\{ \lambda t_0-\mu t_1=\lambda\mu t_2-\lambda^2t_1-\mu^2 t_3=0\right\},
$$ 
for $y=(\lambda,\mu)\in \PP^1$, that we met in \eqref{eq:defining}.
As previously, let $Y_0$
be the open subset of $\PP^1$ where $\mu\neq 0$. 
Every point of $V(\QQ)$ lies on precisely one line $V_y$, for $y\in Y_0(\QQ)$, so that 
$$
N(V;B)=
\sum_{y\in Y_0(\QQ)} 
N(V_y;B).
$$
We have 
$$
N(V_y;B)
=\frac{1}{2}
\#\left\{\t\in \ZZ_{\mathrm{prim}}^4\cap V_y : (t_0,t_1)\neq (0,0), ~\|\t\|\leq B\right\},
$$
where $\|\t\|=\sqrt{t_0^2+\dots+t_3^2}$.
The next result is concerned with an explicit parameterisation of the 
lines $V_y$.

\begin{lemma}\label{lem:lines}
For $\mu\neq 0$ we have 
$$
N(V_y;B)=
\frac{1}{2}
\# \left\{ 
(\tau_0,\tau_1)\in \ZZ_{\mathrm{prim}}^2:
 \begin{array}{l}
\tau_0\neq 0\\
\|(\mu^2\tau_0,\lambda\mu\tau_0, 
\lambda^2\tau_0+\mu\tau_1, \lambda \tau_1)\|\leq B
\end{array}{}
\right\}.
$$
\end{lemma}

\begin{proof}
Suppose first that 
 $\lambda=0$. In this case $V_y$ is the line $t_1=t_3=0$ and the statement of the lemma is clear. 
For the remaining values of $\lambda,\mu$ we deduce from 
the first equation defining $V_{y}$  that 
$$
t_0=h\mu, \quad t_1=h\lambda,
$$
for a non-zero integer $h$, 
since $\gcd(\lambda,\mu)=1$.
Making this substitution into the second equation defining $V_y$, we obtain
\begin{equation}\label{eq:*}
\lambda\mu t_2-\mu^2t_3-h\lambda^3=0.
\end{equation}
It follows from this that $\mu\mid h$ and  $\lambda\mid t_3$.
Thus we may make the change of variables
$$
h=\mu\tau_0, \quad t_3=\lambda \tau_1, \quad t_2=\tau_2,
$$
for $\tau_0,\tau_1\in \ZZ$ such that $\tau_0\neq 0$.
On substituting these into \eqref{eq:*} and dividing through by $\lambda\mu$, this leads to 
$\tau_2=\lambda^2\tau_0+\mu\tau_1$.
We therefore arrive at the parameterisation  in the statement of the lemma.

Since $\gcd(\lambda,\mu)=1$,  in order to complete the proof of the lemma, 
it will suffice to show that  $\t$ is primitive if and only if 
$\gcd(\tau_0,\tau_1)=1$.  But  $\t$ is primitive if and only if 
$\gcd(\mu\tau_0,\tau_2,\lambda\tau_1)=1$, i.e.\ if and only if 
$\delta_1=\delta_2=\delta_3=1$, where
$$
\delta_1=\gcd(\tau_0,\tau_2,\lambda), \quad 
\delta_2=\gcd(\tau_0,\tau_2,\tau_1), \quad 
\delta_3=\gcd(\mu,\tau_2,\tau_1).
$$
Clearly  $\delta_2=\gcd(\tau_0,\tau_1).$
It will therefore suffice to show that $\delta_1=\delta_3=1$ when $\delta_2=1$. But 
\begin{align*}
\delta_1&\mid \gcd(\tau_0,\tau_2,\lambda,\mu\tau_1)=\gcd(\delta_2,\lambda),\\
\delta_3&\mid \gcd(\mu,\tau_2,\tau_1,\lambda^2\tau_0)
=\gcd(\delta_2,\mu),
\end{align*}
from which the claim follows. 
\end{proof}

It is clear that $N(V_y;B)=0$ unless  $|\lambda|,|\mu|\leq \sqrt{B}$.
The region in this counting function is an ellipsoid  which is contained in the region
$$
 \tau_0\ll 
\frac{B}{
\lambda^2+\mu^2}, \quad 
\tau_1\ll 
\frac{B}{\max\{|\lambda|,|\mu|\}}.
$$
Let $N^*(V_y;B)$  be the cardinality in Lemma \ref{lem:lines}, 
in which 
 the coprimality condition $\gcd(\tau_0,\tau_1)=1$ is dropped.
Then 
\begin{align*}
N(V_y;B)
=\frac{1}{2}
\sum_{k\ll B/(\lambda^2+\mu^2)}
\mu(k) N^*(V_y;B/k).
\end{align*}
We may 
 approximate 
$N^*(V_y;B)$  by the volume of the region to within  an error of 
$O(B/\max\{|\lambda|, |\mu|\}+1)$.
This gives
\begin{align*}
N(V_y;B)
&=
\frac{c_{y} B^2}{2}
\sum_{k\ll B/(\lambda^2+\mu^2)}
\frac{\mu(k)}{k^2} 
 +
 O\left(\frac{B}{\max\{|\lambda|, |\mu|\}}+1\right),
\end{align*}
where
$c_{y}$ is the volume of the region
$$
\{(\xi,\eta)\in \RR^2:
(\lambda^2+\mu^2)\xi^2+2\lambda^2\mu \xi\eta+(\lambda^4+\lambda^2\mu^2 +\mu^4)\eta^2\leq 1\}.
$$
The associated  discriminant  is 
$$
4\left\{(\lambda^2+\mu^2)(\lambda^4+\lambda^2\mu^2 +\mu^4)-
(\lambda^2\mu)^2\right\}=4f(\lambda,\mu),
$$
in the notation of Theorem \ref{t:main}, whence 
$
c_{y}=\pi/\sqrt{f(\lambda,\mu)}.
$
Extending the sum over $k$ to infinity we are therefore led to 
an expression for  $N(V_y;B)$, with error term
$O(B/\max\{|\lambda|, |\mu|\}+1)$ and a main term 
equal to 
$$
\frac{\pi^2 B^2}{2\zeta(2)\sqrt{f(\lambda,\mu)}}.
$$
Once summed over $|\lambda|,|\mu|\leq \sqrt{B}$ this error term makes the satisfactory overall contribution $O(B^{3/2}\log B)$. 
Finally, 
 we extend the summations over $\lambda,\mu$ to infinity
 to arrive finally at the statement of  Theorem \ref{t:main}.

\section{Second approach: using Poisson summation}\label{s:second}

In this section we will study the counting function $N(V;B)$ using the methods 
 of Chambert-Loir and Tschinkel \cite[\S 3]{CL-T}.
Let $G$ denote the commutative algebraic group given by the equation 
$
xy-y^3-z=0, 
$
with identity $(0,0,0)$ and addition rule given by 
$$
(x,y)+ (x',y')=(x+x',y+y'+3xx').
$$
This group is isomorphic to $\mathbb{G}_a^2$ and we will view it as such.
There is a $G$-action 
$G\times W \to W$
 given by 
$$
(x,y)\cdot (\t)\mapsto (t_0,t_1+yt_0,t_2+xt_0+3yt_1,t_3+xt_1+yt_2+(xy-y^3)t_0).
$$
One can check that $W$ is an equivariant compactification of $G$.

If we put $(x, y, z)=(t_1/t_0, t_2/t_0, t_3/t_0)$,  then we may regard $V$  as the affine cubic in $\AA^3$ given by the equation  $xy- y^3 - z$. 
We identify any point $(x,y,z)\in G$
with 
a point $(1,x,y,z)\in V$. 
We are then interested in the analytic properties of the height zeta function 
$$
Z(s)=\sum_{(x,y,z)\in G(\QQ)
} H(x,y,z)^{-s}=\sum_{
P=(x,y)\in \mathbb{G}_a^2(\QQ)}
H(P)^{-s},
$$
for $\Re(s)\gg 1$, where  
for $P=(x,y)\in \mathbb{G}_a(\QQ)$, we have
$$
H(P)=H_\infty(x,y)\prod_p H_p(x,y),
$$ 
with
$$
H_v(x,y)=\begin{cases}
\sqrt{1+x^2+ y^2+ (xy-y^3)^2}, & \mbox{if $v=\infty$,}\\
\max\{1, |x|_p ,
 |y|_p, |xy-y^3|_p\},
& \mbox{if $v=p$.}
\end{cases}
$$
Define the local characters $\psi_v:\mathbb{G}_a(\QQ_v)\to \CC^*$ via
$$
\psi_v(x_v)=\begin{cases}
e(-x_v), &\mbox{if $v=\infty$,}\\
e(x_v), &\mbox{if $v=p$}.
\end{cases}
$$
The product of these gives a global 
character
 $\psi: \mathbb{G}_a(\AA_\QQ)\to \CC^*$.
 
 Let $\mu_p$ be the  Haar measure on  $\QQ_p^2$ normalised so that  $\mu(\ZZ_p^2)=1$.
Let $\mu_\infty$ denote the ordinary Lebesgue measure on $\RR$. Then it follows from the Poisson summation formula (see Thm.~2.5 and Prop.~2.6 of \cite{CL-T}) that 
$$
Z(s)=\sum_{\a=(a_1,a_2)\in \mathbb{G}_a^2(\ZZ)} \widehat H(s;\a),
$$
where
\begin{align*}
\widehat H(s;\a)
&=\prod_v \int_{(x,y)\in \mathbb{G}_a^2(\QQ_v)} 
\frac{\psi_v(a_1x+a_2y)}{
H_v(x,y)^{s}} \d \mu_v(x,y)=\prod_v
\widehat H_v(s;\a),
\end{align*}
say.  We will use the notation $\d x \d y$ for $\d \mu_v(x,y)$.
As remarked in the introduction we  will find that the main contribution  comes from the (not all  trivial)  characters corresponding to 
$a_1=0$.

\subsection{Calculation of $\widehat H_\infty(s;\a)$}

We have 
\begin{align*}
\widehat H_\infty(s;\a)
&=\int_{(x,y)\in \RR^2} \frac{e(-a_1x-a_2y)\d x \d y}{
(1+x^2+y^2+(y^3-xy)^2)^{s/2}}.
\end{align*}
This is absolutely convergent for $\Re (s)\geq 2$. 
In fact, for $\Re (s)\geq 2$, repeated integration by parts shows that 
$
\widehat H_\infty(s;\a) \ll_{\sigma,N}	 (1+|\a|)^{-N},
$
for any $N\in \NN$
When $a_1=0$ and $s=2$ we may carry out the integration over $x$ to conclude that 
\begin{align*}
\widehat H_\infty(2;0,a_2)
&= 2\pi \int_{0}^\infty  \frac{\cos(2\pi a_2y) \d y}{
\sqrt{y^6+y^4+2y^2+1}}.
\end{align*}

\subsection{Calculation of $\widehat H_p(s;\a)$ with $a_1\neq 0$}

Suppose  that $a_1\neq 0$.
We are interested in discovering precisely when the  Euler product 
$\widehat H (s;\a)=\prod_p\widehat H_p(s;\a)$
has a pole at $s=2$.
Note that   $H_p(x,y)=1$ if and only if  $(x,y)$ belongs to $ \ZZ_p^2$. Hence
we have
\begin{align*}
\widehat H_p(s;\a)&=\int_{(x,y)\in \QQ_p^2} H_p(x,y)^{-s}
e( a_1x+a_2 y)
\d x\d y\\
&=1+\sum_{j\geq 1} p^{-js}\int_{ \{(x,y)\in \QQ_p^2\,:\, 
\max\{  |x|_p,|y|_p, |xy-y^3|_p\}=p^j \}} 
\hspace{-0.8cm}
e( a_1x+a_2 y)\d x \d y.
\end{align*}
When $x=p^{-j_1}x'$ and  $y=p^{-j_2}y'$ with $x',y'\in U_p$, it is easy to see that
$$
|xy-y^3|_p
=
\begin{cases}
p^{j_1+j_2}, & \mbox{if $j_1>2j_2$},\\
p^{3j_2} & \mbox{if $j_1<2j_2$},\\
p^{3j_2}| x'-y'^2|_p & \mbox{if $ j_1=2j_2$}.
\end{cases}
$$
We let  $S_1(s;\a)$, $S_2(s;\a)$ and $S_3(s;\a)$ denote the contribution from these different cases 
to the sum $\widehat H_p(s;\a)$.

In order to proceed it  will be useful to note that 
\begin{align*}
\int_{U_p} e\left(\frac{cx}{p^j}\right) \d x 
&=
\int_{\ZZ_p} e\left(\frac{cx}{p^j}\right) \d x -
\frac{1}{p}\int_{\ZZ_p} e\left(\frac{cx}{p^{j-1}}\right) \d x \\
&=
\begin{cases}
0, &\mbox{if $j-v_p(c)\geq 2$,}\\
-1/p, &\mbox{if $j-v_p(c)= 1$,}\\
1-1/p, &\mbox{if $j-v_p(c)\leq 0$,}
\end{cases}
\end{align*}
for any $c,j\in \ZZ$.
Note, furthermore,  that we always have the trivial bound
\begin{equation}\label{eq:triv}
|\widehat H_p(s;\a)|\leq 1+O\left(\frac{1}{p^{\sigma-1}}\right),
\end{equation}
which comes from our calculation of $\widehat H_p(s;\0)$.

If $p\mid a_1$ we use \eqref{eq:triv}. Otherwise, supposing that 
$p\nmid a_1$, 
it suffices to calculate
\begin{equation}\label{eq:S1-a}
S_1(s;\a)=\sum_{\substack{j_1\geq 1\\j_1>2j_2\\j_2\geq 0}} p^{(j_1+j_2)(1-s)}I(j_1,j_2)+
\sum_{\substack{j_1\geq 1\\ j_2< 0}} p^{-j_1s+j_1+j_2 }I(j_1,j_2),
\end{equation}
where
\begin{align*}
I(j_1,j_2)
&=\int_{U_p^2} e\left(\frac{a_1x}{p^{j_1}}+
\frac{a_2y}{p^{j_2}} \right)\d x\d y
=
\begin{cases}
0, &\mbox{if $j_1\geq 2$,}\\
-1/p(1-1/p), &\mbox{if $j_1=1, ~j_2\leq 0$.}
\end{cases}
\end{align*}
A simple computation now reveals that 
$S_1(s;\a)=-p^{-s}$. Hence we conclude that 
$\widehat H (s;\a)$ is absolutely convergent and bounded by $O(|\a|^\ve)$ for any $\ve>0$, provided that 
$\Re(s)>3/2$ and $a_1\neq 0$.

\subsection{Calculation of $\widehat H_p(s;0,a_2)$}

Next we suppose that $\a=(0,a_2)$.
It will be convenient to set $\alpha=v_p(a_2)\geq 0$, with the convention that $\alpha=\infty$ if $a_2=0$.
In this case it follows from \eqref{eq:S1-a} that
\begin{align*}
S_1(s;0,a_2)
=~&-\sum_{\substack{j_1\geq 2\alpha+3}} p^{(j_1+1-\alpha)(1-s)-1}(1-1/p)\\
&+\sum_{\substack{j_1>2j_2\\ 0\leq j_2\leq \alpha}} p^{(j_1+j_2)(1-s)}(1-1/p)^2
+
\frac{p^{1-s}(1-1/p) }{p(1-p^{1-s}) }\\
=~&
\frac{p^{1-s}(1-1/p)(1-p^{3(\alpha+1)(1-s)})
(1-p^{2-3s})
}{(1-p^{1-s})(1-p^{3(1-s)})},
\end{align*}
since now
\begin{align*}
I(j_1,j_2)
&=
\begin{cases}
0, &\mbox{if $j_2\geq 2+\alpha$,}\\
-1/p(1-1/p), &\mbox{if $j_2=1+\alpha$,}\\
(1-1/p)^2, &\mbox{if $j_2\leq \alpha$.}
\end{cases}
\end{align*}
In particular we have 
$$
S_1(2;0,a_2)=\frac{(1-p^{-3\alpha-3})(1-p^{-4})}{p(1-p^{-3})}.
$$
Next 
\begin{align*}
S_2(s;0,a_2)
=~&
-p^{-3(1+\alpha)(s-1)-2}+
\sum_{\substack{1\leq j_2\leq \alpha}} p^{-3j_2(s-1)-1} (1-1/p)\\
=~&
-p^{-3(1+\alpha)(s-1)-2}+\frac{p^{2-3s}(1-1/p)(1-p^{3\alpha(1-s)})}{1-p^{3(1-s)}}\\
=~&
-p^{-5-3\alpha}+\frac{(1-1/p)(1-p^{-3\alpha})}{p^4(1-p^{-3})}.
\end{align*}

To calculate $S_3(s;0,a_2)$, 
it will be convenient to put
$$
\delta_j=
\begin{cases}
0, &\mbox{if $j=1+\alpha$,}\\
1, &\mbox{if $j\leq \alpha$}.
\end{cases}
$$
Let 
 $T(h)$ denote the set of 
$(x,y)\in U_p^2$ such that 
$|x-y^2|_p=p^{-h}$. Then 
$$
\int_{T(h)}e\left(\frac{a_2 y }{p^{j_2}}\right) \d x \d y
=
\begin{cases}
0& \mbox{if  $j_2\geq 2+\alpha$,}\\
(\delta_{j_2}-1/p)(1-1/p) p^{-h}& \mbox{if  $h\geq 1$, $j_2\leq 1+\alpha$,}\\
(\delta_{j_2}-1/p)(1-2/p)& \mbox{if $h=0$, $j_2\leq 1+\alpha$.}
 \end{cases}
$$
Writing $S_3(s;\a)=A_\a(s)+B_\a(s)+C_\a(s)$, we see that
\begin{align*}
A_\a(s)=~&
\sum_{\substack{h\geq j_2\\
1\leq j_2\leq 1+\alpha}}
 p^{-2(1+\alpha) s+3(1+\alpha)-h}(\delta_{j_2}-1/p)(1-1/p) \\
=~&
-p^{-2(1+\alpha)(s-1)-1}
+\sum_{1\leq j\leq \alpha} p^{-2j_2(s-1)}(1-1/p)\\
=~&-p^{-2(1+\alpha)(s-1)-1}+\frac{p^{2(1-s)} (1-1/p)(1-p^{2\alpha(1-s)})}{1-p^{2(1-s)}
},
\end{align*}
whence 
$$
A_\a(2)=-p^{-3-2\alpha}+\frac{(1-1/p)(1-p^{-2\alpha})}{p^2(1-p^{-2})}.
$$
Next 
\begin{align*}
B_{\a}(s) =~&
\sum_{
\substack{
1\leq h\leq j_2-1\\
1\leq j_2\leq \alpha+1
}}
p^{-(3j_2-h)(s-1)}(\delta_{j_2}-1/p)(1-1/p)\\
=~&-\frac{p^{(3+2\alpha)(1-s)-1}(1-p^{\alpha(1-s)})(1-1/p)}{1-p^{1-s}}\\
~&+
\sum_{
1\leq j_2\leq \alpha}
\frac{p^{3j_2(1-s)}(1-1/p)^2(p^{(j_2-1)(s-1)}-1)}{1-p^{1-s}}
\\
=~&
-\frac{p^{(3+2\alpha)(1-s)-1}(1-p^{\alpha(1-s)})(1-1/p)}{1-p^{1-s}}\\
&+
\frac{p^{3(1-s)}(1-1/p)^2}{1-p^{1-s}}\left(
\frac{ 1-p^{2\alpha(1-s)}}{1-p^{2(1-s)}}-
\frac{ 1-p^{3\alpha(1-s)}}{1-p^{3(1-s)}}
\right).
\end{align*}
Hence
$$
B_\a(2)=
-p^{-4-2\alpha}(1-p^{-\alpha})+
p^{-3}(1-1/p)
\left(
\frac{ 1-p^{-2\alpha}}{1-p^{-2}}-
\frac{ 1-p^{-3\alpha}}{1-p^{-3}}
\right).
$$
Finally, we have
\begin{align*}
C_{\a}(s)
=~&
-(1-2/p)p^{3(1+\alpha)(1-s)-1}
+\sum_{1\leq j_2\leq \alpha
}
p^{-3j_2(s-1)}(1-1/p)(1-2/p)\\
=~&
(1-2/p)p^{3(1-s)}\left(-p^{3\alpha(1-s)-1}+\frac{(1-1/p)(1-p^{3\alpha(1-s)})}{1-p^{3(1-s)}}\right).
\end{align*}

Putting this  together, we see that
\begin{align*}
\widehat H_p(2;0,a_2)&=
1+S_1(2;\a)+S_2(2;\a)+S_3(2;\a)\\
&=
\left(1+\frac{1}{p}+\frac{1}{p^2}\right)\left(1-\frac{1}{p^{2+2\alpha}}\right).
\end{align*}

\subsection{Conclusion}
We have
$
Z(s)=
Z_1(s)+Y(s)
$
where $Y(s)$
is holomorphic and bounded for $\Re(s)>3/2$ and 
$$
Z_1(s)=
\sum_{m\in \ZZ}
 \widehat H(s;0,m),
\quad 
\widehat H(s;0,m)
=\prod_v
\widehat H_v(s;0,m).
$$
Our work shows that 
$\hat H(s;0,m)=\zeta(s-1)E_m(s)$, where $E_m(s)$ is holomorphic and bounded for $\Re(s)\geq 2$. Furthermore, 
$
E_0(2)=\zeta(3)^{-1}\widehat H_\infty(2;\0)
$
and 
$$
E_m(2)=\frac{\sigma_{-2}(m)}{\zeta(2)\zeta(3)} \widehat H_\infty(2;0,m) \quad (m\neq 0),
$$
where
$
\sigma_{-2}(m)=\sum_{d\mid m} d^{-2}.
$
We extend the latter function to all of $\ZZ$ by setting $\sigma_{-2}(0)=\zeta(2)$.
Finally, we  recall that 
$$
\widehat H_\infty(2;0,m)
= 2\pi \int_{0}^\infty  \frac{\cos(2\pi my) \d y}{
\sqrt{y^6+y^4+2y^2+1}}.
$$

A standard Tauberian theorem (see Tenenbaum 
\cite[\S II.2]{ten}, for example) 
therefore gives an asymptotic formula of 
the shape  $N(V;B)= c B^2+O(B^{\theta})$, for any $\theta>3/2$, 
with 
\begin{align*}
c
&=\frac{1}{2}\sum_{m\in \ZZ} 
E_m(2)=\frac{\pi}{\zeta(2)\zeta(3)}
\sum_{m\in \ZZ} \sigma_{-2}(m)
\int_{0}^\infty  \frac{ \cos (2\pi m y)\d y}{
\sqrt{y^6+y^4+2y^2+1}}.
\end{align*}
In order to show that this is compatible with 
Theorem \ref{t:main} we need to prove that
\begin{align*}
\frac{1}{\zeta(3)}\sum_{m\in \ZZ}\sigma_{-2}(m)
\int_{0}^\infty & \frac{ \cos (2\pi m y )\d y}{
\sqrt{y^6+y^4+2y^2+1}}
=\frac{1}{2} \sum_{\substack{(\lambda,\mu)\in \ZZ_{\mathrm{prim}}^2\\
\mu\neq 0}} \frac{1}{\sqrt{f(\lambda,\mu)}},
\end{align*}
with $f(\lambda,\mu)$ as in the statement of the theorem. 
But this follows from a straightforward application of Poisson summation. Thus, 
using the M\"obius function to detect the condition $\gcd(\lambda,\mu)=1$, we find 
that
\begin{align*}
\frac{1}{2} \sum_{\substack{(\lambda,\mu)\in \ZZ_{\mathrm{prim}}^2\\
\mu\neq 0}} \frac{1}{\sqrt{f(\lambda,\mu)}}
&=
\frac{1}{\zeta(3)}\sum_{v=1}^{\infty} \sum_{u\in \ZZ}\frac{1}{\sqrt{f(u,v)}}\\
&=
\frac{1}{\zeta(3)}\sum_{v=1}^{\infty} \sum_{a\in \ZZ}  \int_{-\infty}^\infty \frac{e(at)\d t}{\sqrt{f(t,v)}}\\
&=
\frac{1}{\zeta(3)}\sum_{v=1}^{\infty} \frac{1}{v^2}\sum_{a\in \ZZ}  \int_{-\infty}^\infty \frac{e(av y)\d y}{\sqrt{f(y,1)}}
\\
&=\frac{1}{\zeta(3)}
\sum_{m\in \ZZ}\sigma_{-2}(m) \int_{0}^\infty \frac{\cos(2\pi m y)\d y}{\sqrt{f(1,y)}},
\end{align*}
as required.

\end{document}